\documentclass[a4paper,10pt,oneside]{article}
\usepackage[english]{ babel}
\usepackage[utf8]{inputenc}

\usepackage[english]{babel}
\usepackage{graphicx}
\usepackage{subfig}
\usepackage{amssymb}
\usepackage{amsmath}
\usepackage{amsthm}
\usepackage{enumerate}
\usepackage{cite}
\usepackage{xcolor}

\newtheoremstyle{case}{}{}{}{}{}{:}{ }{}
\theoremstyle{casep}

\newcommand{\C}{\mathbb{C}}
\newcommand{\Z}{\mathbb{Z}}
\newcommand{\Pj}{\mathbb{P}}

\newcommand{\N}{\mathbb{N}}

\newtheorem{theorem}{Theorem}[section]
\newtheorem{proposition}[theorem]{Proposition}
\newtheorem{lemma}[theorem]{Lemma}

\theoremstyle{definition}
\newtheorem{definition}[theorem]{Definition}

\newtheorem{example}[theorem]{Example}
\newtheorem{remark}[theorem]{Remark}

\newtheorem*{theorem*}{Theorem}

\newcommand{\rank}{\operatorname{rank}}

\usepackage{xcolor}

\begin{document}

\title{On a geometric method for the identifiability of forms.}
\date{}

\author{Andrea Mazzon}


\maketitle
\begin{abstract}
We introduce a new criterion which tests if a given decomposition of a given ternary form $T$ of even degree is unique. The criterion is based on the analysis of the Hilbert function of the projective set of points $Z$ associated to the decomposition, and on the Terracini's Lemma which describes tangent spaces to secant varieties. The criterion works in a range for the length
of the decomposition which is equivalent to the range in which the reshaped Kruskal's criterion (see \cite{sette}) works. Our criterion determines an algorithm for the identifiability of $T$ which is sensibly faster than algorithms based on the reshaped Kruskal's criterion, especially when the set of points $Z$ is not in general position.

\end{abstract}

\section{Introduction}
\label{introduction}

The paper is devoted to the following problem: find an efficient method to determine if a given rank decomposition of a homogenous polynomial (form), seen as a symmetric
tensor, is minimal and unique, so that the tensor is identifiable in the sense of \cite{sette}. In other words, we are looking for algorithms, with low computational cost, which determine when a form $ T $ can be written in a unique way as a minimal sum of powers of linear forms (up to scaling). 

The standard method to determine the identifiability of a tensor, introduced by Kruskal in \cite{tredici}, provides an algorithm that verifies if a given decomposition
of a form $T$ is unique and minimal. Kruskal's algorithm can be implemented in a reasonable computational time. The identifiability of a tensor is relevant for 
applications in signal processing, image reconstruction, artificial intelligence, statistical mixture models, etc.
 (see e.g. \cite{quindici}, \cite{sedici}, \cite{diciassette}). Indeed the Kruskal's algorithm is used extensively by researchers in technological fields
(see e.g. \cite{diciotto}).
 
The main weakness of the Kruskal's algorithm is that it can provide an answer for the identifiability of $T$ only if the number of summands in the decomposition is \emph{small}, compared with the range of possible ranks. This means that Kruskal's method applies only in very specific situations.

In order to broaden the range of applicability of Kruskal's type methods, in \cite{sette} Chiantini, Ottaviani and Vannieuwenhoven introduced the \emph{reshaped Kruskal's method}. The reshaped Kruskal's method
applies in a wider range, with respect to the original Kruskal's method, but it requires the computation of some geometric invariants of a decomposition $A$ of $T$ (see Definition \ref{remark:span} below).
The computational cost of finding these invariants of $A$ is still reasonable, as soon as $A$ is sufficiently general. On the other hand,
in special cases, mainly when $A$ is in some \emph{special position}, the computational cost increases oddly.

The aim of this note is to investigate new methods to determine the uniqueness and minimality of a decomposition of $T$,
 which can produce a considerable cut in the computational cost of the procedure, with respect to the reshaped Kruskal's method, 
 even when $A$ is not in general position.

Our new method is based on geometric properties of the decomposition $A$ of $T$, viewed as a finite set of points in a projective space.
In a range for the length of $A$ which is similar to the range where the reshaped Kruskal's criterion applies,
we will show that the analysis of the Hilbert function of $A$ and its \emph{Cayley-Bacharach properties} (see the definitions below) produces
a criterion which determines that either $A$ is unique, or there exists an infinite family of
decompositions of $T$, containing $A$. Next, we show that the existence of the infinite family can be excluded by the analysis
of the Terracini's tangent space to the secant variety to a Veronese variety (a classical geometric object for the study of secant varieties, see \cite{diciannove}). 
To do that, we extend a procedure introduced and discussed in \cite{sette} and \cite{otto}.
Thus, summing up the two procedures, we get a new criterion to determine the identifiability of $T$, which is described in Section 4 below.

The criterion, which is based on arguments of Algebraic Geometry but requires only algorithms of Linear Algebra, 
produces an effective method to test the uniqueness (and thus also the minimality) of a decomposition, whose computational cost can be considerably lower than the cost of the reshaped Kruskal's method.

We analyse in particular the procedure for the case of ternary forms $T$ (i.e. symmetric tensors of type $3\times \dots \times 3$), for which
we know a decomposition $A$, corresponding to a set of points in $\Pj^2$. This is the first case
in which the new methods apply and reduces considerably the computational cost of the procedure, see the final Remark \ref{compucost}.
Of course, a similar analysis applies in more generality, with adjustments for any specific case.

We observe that the new method essentially suggests that, in many cases, one can substitute the computation of the higher
Kruskal's rank with the computation of the dimension of the Terracini's space, which turns out to be much cheaper in terms of computational cost.
From this point of view, the paper has been inspired by, and it is a sort of continuation of 
\cite{due}, \cite{ventuno}, section 6 of \cite{sette}, and \cite{otto}.

\section{Notation and preliminaries}
\label{Notation and preliminaries}
In this section we recall some useful results that we will use for the investigation of the identifiability of forms (i.e. for the identifiability of symmetric tensors).
\smallskip

We work over the complex field.

Let $T$ be a form of degree $d$ in $3$ variables, i.e. $T\in \Pj(Sym^d(\C^{3}))$. A \emph{Waring decomposition} of $T$ is an expression
$$T= a_1T_1+\dots + a_rT_r$$
where each $T_i$ is the $d$-th power of a linear form $T_i=L_i^d$. In principle, since we are working over an algebraically closed field, 
we could get rid of the coefficients $a_i$'s. We will maintain them because our starting point will be the set of linear forms $L_i$, and the various 
forms $T$ that can be decomposed by the fixed $L_i$'s will thus be obtained by changing the choice of the coefficients.

Since we will study the geometry of the decomposition, we take the projective point of view. So, by abuse of notation, we will identify each linear form $L_i$ with a point
$P_i$ in a projective plane $\Pj^2=\Pj(\C^3)$. The power $L_i^d$ corresponds to the image of $P_i$ via the Veronese map $\nu_d$ of degree $d$ which sends
$\Pj^2$ to the projective space $\Pj(Sym^d(\C^{3}))$ of forms of degree $d$. We will also identify $T$ with a point of $\Pj(Sym^d(\C^{3}))$, and, by abuse of notation, we will denote
by $T$ both the form in $Sym^d(\C^{3})$ and the point in $\Pj^N$ which represents $T$, where $N = \binom{3+d}d - 1 $. 

So, the Waring decomposition above identifies a finite subset $A=\{P_1,\dots,P_r\}$ of $\Pj^2$.
\smallskip

For any finite subset $Z$ of the projective plane, we denote by $\ell(Z)$ the cardinality of $Z$.

With this notation, we give the following \emph{geometric} definition of decomposition, which is nothing more than a rephrasement of a Waring decomposition
in geometric terms.

\begin{definition}
Let $A \subset \Pj^m $ be a finite set, $A=\{P_1,\dots, P_r\}$. $A$ is a \emph{decomposition} of $T\in \Pj(Sym^d(\C^{m+1}))$ if $ T $ belongs to $ \langle v_d(A) \rangle$, 
the linear space spanned by the points of $v_d(A)$. In other words, for a choice of scalars $a_i$'s,
$$T = a_1v_d(P_1)+\dots + a_rv_d(P_r). $$
The number $\ell(A)=r$ is the \emph{length} of the decomposition.

The decomposition $A$ is \emph{non-redundant} if $T$ is not contained in the span of $v_d(A')$, for any proper subset
$A'\subset A$. In particular, if $v_d(A)$ is linearly dependent, then $A$ cannot be non-redundant.

We say that $A$ is \emph{minimal} if no decompositions of $T$ have length smaller than $r$.
We say that $T$ is \emph{identifiable} if it has a unique minimal decomposition. It is almost obvious that if $A$ 
is the unique decomposition of $T$ of length $r$, then $A$ is minimal and $T$ is identifiable.
\end{definition}

In order to introduce the reshaped Kruskal's criterion, with respect to which we will compare our algorithm, 
we need the notion of \emph{Kruskal's rank}. This notion, which in special cases is different from the usual notion of rank, 
has been introduced by Kruskal for matrices. We rephrase it in the geometric language, for sets of points in a projective space.

\begin{definition}\label{kr}
For a finite set $Z \subset \Pj^m$, the \emph{Kruskal's rank} $k(Z)$ of $Z$ is the maximum $k$ for which any subset
of cardinality $\leq k$ of $Z$ is linearly independent.

The Kruskal's rank $k(Z)$ is bounded above by $m+1$ and $\ell(Z)$. 
\end{definition}

\begin{remark}\label{ksub} It is a consequence of the irreducibility of projective spaces that for a \emph{sufficiently general} subset 
$Z\subset\Pj^m$ all the Kruskal's ranks $k_d(Z)$ are maximal and coincides with the rank of a matrix whose rows are
projective coordinates for the points of $Z$.

The Kruskal's rank $k(Z)$ attains the maximum $\min\{m+1, \ell(Z)\}$ when all the subsets of $Z$ of cardinality at most $m+1$ are linearly independent. In this case, for any subset $Z'\subset Z$ one has
$k(Z')=\min\{m+1, \ell(Z')\}$.

Notice also that for any subset $Z'\subset Z$, we have $k(Z')\geq \min\{\ell(Z'), k(Z)\}$.
\end{remark}

We can use the Veronese maps to define the higher Kruskal's ranks of a finite set $Z$.

\begin{definition}\label{remark:span}
For a finite set $Z \subset \Pj^m$, the \emph{$d$-th Kruskal's rank} $k_d(Z)$ of $A$ is the Kruskal's rank of the image of $Z$ via the Veronese map $v_d$.
Thus the $d$-th Kruskal's rank $k_d(Z)$ is bounded by $\min\{\ell(Z), \binom{m+d}d\}$.

The Kruskal's rank $k(Z)$ coincides with the first Kruskal's rank $k_1(Z)$.

We notice that finding $k_d(A)$ is equivalent to finding the dimension of the subspace spanned by $v_d(Z')$ for the subsets $Z'$ of $Z$.
\end{definition}

\vspace{5mm}

For a finite set of points $Z$ let $I_Z$ be the homogeneous ideal associated to $Z$. We will denote with $I_Z(d)$ the homogeneous part of degree $d$ of $I_Z$. 
For all $d$, $I_Z(d)$ is a finite-dimensional vector space over the field $\C$.

We recall that the Hilbert function $h_Z$ of a set of points $Z \subset \mathbb{P}^{m}$ is defined as follows:
	$$h_Z(d)= \dim(Sym^d(\mathbb{C}^{m+1}))- \dim (I_Z(d))= \binom{m+d}{d} - \dim(I_Z(d)).$$

We recall also that the first difference of Hilbert function $Dh_Z$ of $Z$ is defined as:$$Dh_Z(j)=h_Z(j)-h_Z(j-1), \ \ j\in\mathbb{Z}.$$

Equivalently, we give an alternative definition of the Hilbert function which can be more natural for people working in applicative fields. First of all, we define the evaluation map.

\begin{remark}
Let $Y\subset \C^{m+1} $ be an ordered, finite set of cardinality $\ell $ of vectors. Fix an integer $ d \in \N $. 

The \emph{evaluation map of degree $d$ on $Y$} is the linear map
$$ ev_{Y}(d): Sym^d(\C^{m+1}) \to \C^\ell $$ 
which sends $ F \in Sym^d(\C^{m+1}) $ to the evaluation of $ F$ at the vectors of $Y$. 

Let $Z \subset \Pj^m $ be a finite set. Choose a set of homogeneous coordinates for the points of $Z$.
We get an ordered set of vectors $Y\subset \C^{m+1} $, for which the evaluation map $ev_{Y}(d)$ is defined for every $d$.

If we change the choice of the homogeneous coordinates for the points of the fixed set $Z$, the evaluation map
changes, but all the evaluation maps have the same rank. 
So, we can define the \emph{Hilbert function} of $Z$ as the map:
$$ h_Z : \Z \to \N \qquad h_Z(d) = \rank(ev_{Y}(d)) .$$ 
\end{remark}

We will use several times the well known fact that the first difference of Hilbert function, after a certain point, is not increasing.

\begin{proposition}\label{nonincr}
Let $ Z \subset \Pj^m $ be a finite set of points. Assume that for some $j>0$ we have $Dh_Z(j) \leq j$. Then:
$$ Dh_Z(j) \geq Dh_Z(j+1), $$
so that $ Dh_Z(j) \geq Dh_Z(i) $ for all $i\geq j$.
In particular, if for some $j>0$,we have $Dh_Z(j)=0 $, then $Dh_Z(i)=0$ for all $i\geq j$.
\end{proposition}
\begin{proof} See Section 3 of \cite{tre}.

\end{proof}

The following lemma collects some well known facts about the Hilbert function and its first difference. 

The proofs of this properties can be found in the literature (see e.g. the book of Iarrobino and Kanev \cite{ventitre}) but they are quite sparse. All the proofs are collected in Lemma 2.16 and Proposition 2.17 of \cite{uno}.

\begin{lemma} Let $ Z \subset \Pj^m$ be a finite set of points and set $\ell= \ell(Z)$. Then we have:
\begin{itemize}
\item[1)] $h_Z(d) \leq \ell$ for all $d$;
\item[2)] $Dh_Z(d)=0$ for $d<0$;
\item[3)] $h_Z(0) = Dh_Z(0) = 1$; 
\item[4)] $Dh_Z(d) \geq 0 $ for all $d$;
\item[5)] $h_Z(d) = \ell(Z)$ for all $d \geq \ell(Z)- 1$;
\item[6)] $h_Z(i) = \sum_{0\leq d\leq i} Dh_Z(d)$;
\item[7)] $Dh_Z(d) = 0$ for $d >> 0$ and $\sum_d Dh_Z(d) = \ell(Z)$;
\item[8)] if $h_Z(d) = \ell(Z), $then $Dh_Z(d + 1) = 0$;
\item[9)] if $Z' \subset Z$, then, we have $h_{Z'} (d) \leq h_Z(d)$ and $Dh_Z' (d) \leq Dh_Z(d)$ for every $d \in \Z$.
\end{itemize}
\label{lemma:mix}
\end{lemma}

 The next proposition gives us an useful information about the first difference Hilbert function of the union of two different decompositions of a form $T$.

\begin{proposition} Given two different decompositions $A$, $B$ of a form $T$ of degree $ d $ in $ m+1 $ variables, then
$h_{A\cup B} (d) < \ell(A \cup B)$, so $Dh_{A \cup B} (d +1) > 0$.
\label{d+1}
\end{proposition}
\begin{proof} See Lemma 1 of \cite{BallBern12a}.

\end{proof}

The shape of the first difference Hilbert function $ Dh_Z $ gives us some information on how the points of $ Z $ are located in the plane. In particular, we cite the following Theorem of Bigatti Geremita and Migliore. 

\begin{theorem} \label{BGM} Let $ Z \subset \mathbb{P}^m$ be a finite set. Assume also that for some $ s \leq j$, $Dh_Z(j) = Dh_Z(j +1) = s $. Then there exists a reduced curve $C$ of degree $s$ such that, setting $Z'= Z \cap C$ and $Z'' = Z \setminus Z'$:
	\begin{itemize}
		\item for $i \geq j-1, \ h_{Z'}(i)=h_Z(i)$
		\item for $i \leq j, \ h_{Z'}(i)=h_C(i)$
		\smallskip
		\item $Dh_{Z'} = \bigg \{
		\begin{array}{rl}
		Dh_C(i) & for \ i \leq j+1 \\
		Dh_Z(i) & for \ i \geq j \\
		\end{array}$
		
	\end{itemize}
	In particular, $Dh_{Z'} (i) = s$ for $s \leq i \leq j + 1$.

\end{theorem}
\begin{proof} See Theorem 3.6 of \cite{tre}.

\end{proof}

Another property that we will use is the \emph{Cayley-Bacharach} property.

\begin{definition}\label{CB}
A finite set $Z\subset \Pj^m$ satisfies the \emph{Cayley-Bacharach property in degree $d$}, 
abbreviated as $CB(d)$, if for any $P \in Z$ every form of degree $d$ vanishing at $ Z\setminus\{ P\}$
also vanishes at $P$.
\end{definition}

The Cayley-Bacharach property gives us a lot of information about the shape of the first difference of Hilbert function. We will mainly use the the following.

\begin{theorem}\label{GKRext}
If a finite set $ Z \subset \Pj^{m} $ satisfies $\mathit{CB}(i)$, then for any $ j $ such that $ 0 \leq j \leq i+1 $ we have
$$ Dh_{Z}(0)+Dh_{Z}(1)+\cdots + Dh_{Z}(j) \leq Dh_{Z}(i+1-j)+\cdots +Dh_{Z}(i+1).$$
\end{theorem}
\begin{proof} See Theorem 4.9 of \cite{otto}.

\end{proof}

An important case where Cayley-Bacharach holds is the following:

\begin{lemma}\label{CBdis}
Let $T$ be a form of degree $ d $ in $ m+1 $ variables and consider two non-redundant decompositions $A, B$ of $T$. 
Set $Z= A \cup B$. If $A \cap B = \emptyset$, then $Z$ has the Cayley-Bacharach property $CB(d)$.
\end{lemma}
\begin{proof} See Lemma 5.3 of \cite{nove}.

\end{proof}

We conclude this section by recalling an important tool used to prove identifiability of forms: the Reshaped Kruskal's criterion. 
We point out that Kruskal's theorem gives us a method to determine the identifiability of a form which is effective
(in the sense of \cite{sette}) but it can be expensive from the point of view of computational costs, especially when the decomposition $A$ is not in general position. 

\begin{theorem} \label{K} ({\bf Reshaped Kruskal's criterion}) Let $T$ be a form of degree $d$ (in any number of variables) and let $A$ be a non-redundant decomposition of $T$ with $\ell(A)=r$. Fix a partition $a,b,c$ of $d$ and call $k_a,k_b,k_c$ the Kruskal's ranks of $v_a(A),v_b(A), v_c(A)$ respectively.
If: $$ r\leq \frac {k_a+k_b+k_c-2}2,$$
then $T$ has rank $r$ and it is identifiable.
\end{theorem}
\begin{proof} See Section 4 of \cite{dieci}.

\end{proof}

A direct application of this criterion is the following proposition.

\begin{proposition}
Fix $ n \in \N $. Given $T$ a form in three variables of degree $d=8+2n$ such that:

\begin{itemize}
	\item[1)]$A=\{P_1,...,P_r\}\subset\Pj^2$ is a decomposition of $ T $ of length $r=\ell(A)\leq 11+3n$.
	\item[2)] The second Kruskal's rank $ k_2 $ of $A$ is $k_2=\min\{6,r\}$.
	\item[3)] The $(n+3)$-th Kruskal's rank $k_{n+3}$ of $A$ is $k_{n+3} \geq \{r,3n+9\}$.
\end{itemize} 

Then, $T$ is identifiable.
\label{proposition:1}
\end{proposition}
\begin{proof} If $r=1$ there is nothing to prove so assume $r>1$.
We consider the partition $p= 8 +2n =(3+n)+(3+n)+2 $ and we want to prove that $$r \leq \frac{k_{n+3}+k_{n+3}+k_2 -2 }2.$$

Assume that $r=\ell(A)$ is smaller than $3n+10$. Then the third assumption means that the $(n+3)$-th Kruskal rank $k_{n+3}$ of $A$ is $r$. 
Take the partition $d=8+2n=(n+3)+(n+3)+2$. Then we have:
$$\frac {k_{n+3}+k_{n+3}+k_2 -2 }2 = \frac {2r+k_2 -2} 2.$$

Now there are two different cases. If $r>6$ then $$\frac {k_{n+3}+k_{n+3}+k_2 -2 }2 = \frac {2r+6 -2} 2 = r +2 \geq r.$$
thus, by the reshaped Kruskal's criterion, we get that $T$ has rank $r$ and $A$ is the unique decomposition
of $T$.

If $r<6$ then $$\frac {k_{n+3}+k_{n+3}+k_2 -2 }2 = \frac {3r -2} 2 = r + \frac {r -2} 2.$$ 

From $r \geq 2$ we have: $$r + \frac {r -2} 2 \geq r $$
and, as before, by the reshaped Kruskal's criterion, we get that $T$ is a form of rank $r$ and it is identifiable.

 Assume that $r=\ell(A)$ is bigger or equal than $3n+10$. Take the partition $d=8+2n=(n+3)+(n+3)+2$:
 $$ \frac{k_{n+3}+k_{n+3}+k_2-2}{2} \geq \frac{18+6n+k_2-2}{2} $$
 
 In this case $ r \geq 10+3n >6$ and $k_{n+3} \geq 3n+9$ so:
$$\frac {k_{n+3}+k_{n+3}+k_2 -2 }2 \geq \frac {18+6n+6 -2} 2 = 11+3n \geq r$$ so we have that $T$ is identifiable by the reshaped Kruskal's criterion.

\end{proof}

\begin{remark}
In the previous proposition we take as a lower bound for $ k_{n+3} $ the value $ 3n+9 $. This is due to the fact that we consider particularly interesting for our investigation the case in which a decomposition $ A=\{P_1,\dots ,P_r\} \subset \Pj^2 $ of a form $ T $ in three variables is contained in a unique plane cubic curve (so a case in which the points of $ A $ are not in general position). 

In fact, we will see in Remark \ref{costcomp} that this is a situation in which using Kruskal's criterion can be rather demanding from a computational point of view. Thus, this is the situation in which the improvement due to our method is more effective. So, we will focus on the cases in which $ k_{n+3} $ is equal to $ \min\{ r,3n+9 \} $.
\end{remark}
 
\begin{remark}
Suppose that, in the situation of Proposition \ref{proposition:1}, we know that \\ $ k_{n+3}=\min \{ r,3n+9\} $. In this case we observe that $11+3n$ is the maximum value of $r$ for which one can hope to prove the identifiability by using the reshaped Kruskal's criterion.

Indeed, in Proposition \ref{proposition:1} we showed that when $d=8+2n$, if we take the partition $a_1=2$ $a_2=a_3=3+n$
of $d$, then the reshaped Kruskal's criterion tells us that forms of rank $r=11+3n$, whose decomposition satisfies $k_2=\min\{6,r \}$ 
and $k_{n+3}=\min\{r,3n+9 \}$, are identifiable.

We show that there are no partitions for which the reshaped Kruskal's criterion determines the identifiability when $r > 11+3n$.

Indeed, take another partition $b_1$, $b_2$ and $b_3$. We can write each $b_i$ as $b_i=3+a_i$ with $a_i= -2, -1,0,1,2 \dots $.

We claim that $k_{b_i} \leq 9+3a_i $.
Indeed, if $a_i=-2$ then $b_i=1$ and $k_1$ is at most $3$, by definition. For the same reason, if $a_i=-1$ then $b_i=2$ 
and $k_2$ is at most $6$ by definition. 
Suppose $b_i \geq 0$. As we said in remark \ref{remark:span}, we have to find the dimension of $\left\langle v_{b_i}(A) \right\rangle$.
In our case, all the points of $A$ are contained in a cubic curve by hypothesis. Moreover, a cubic curve is a normal elliptic curve and the image of a normal elliptic curve under a Veronese map is again a normal elliptic curve. Thus, $v_d(C)$ generates a projective space of dimension $ \deg(v_d(C))-1= 3d-1.$

So we have: \begin{multline*}\frac{k_{b_1} + k_{b_2} + k_{b_3}-2}{2} \leq \frac{9+9+9 + 3(a_1+a_2+a_3)-2}{2} = \\ =\frac{27+3(2n-1)-2}{2}=11+3n 
\end{multline*}
\end{remark}

\begin{remark} 
	\label{costcomp}
In order to use the reshaped Kruskal's criterion and apply Proposition \ref{proposition:1} to a decomposition $A=\{ P_1,\dots, P_r \} \subset \Pj^2$ of a form $T$ in three variables we need to compute that:
\begin{itemize}
\item $k_2(A)= \min\{ r, 6 \}$;
\item $k_{n+3}(A) \geq \min \{ r,9+3n \} $; 
\end{itemize}
and verify that the inequality $r \leq \frac{k_1+k_2+k_3-2}{2}$ holds.

This is done by determining the rank of matrices derived by the coordinates of the points of some Veronese images of $A$.
The standard method to find the rank of a matrix is the Gauss elimination method. 
The computational cost of computing the rank of a matrix $\C^{m \times n}$ using this method is of $\frac{2}{3} m^2 \cdot n$ flops (see Chapter 3.3 of \cite{sei}). 
In particular, if the matrix is a square matrix $ n \times n $ then the Gauss elimination method has a cost in the order of $ \frac{2}{3} n^3$.

In order to verify that $k_2(A)= \min \{ 6,r \}$ we have to compute the rank of all the $6 \times 6$ sub-matrices of the matrix $[v_2(P_1), \dots , v_2(P_r) ]$.
So we have to find the rank of $\binom{r}{6} \approx \frac{r^6}{6!}$ matrices. By using the Gauss elimination algorithm,
 we see that the computational cost is about
$$\frac{2\cdot 6^3}{3} \cdot \frac{r^6}{6!} = \frac{r^6}{5}.$$ 

In the same way, to verify $k_{n+3}(A) \geq \min \{ 9+3n,r \}$ we may have to compute the rank of all the submatrices $\binom{n+5}{2} \times 3n+9$ 
of the matrix: $$[v_{n+3}(P_1), \dots , v_{n+3}(P_r) ].$$ So, the worst case is when $r=11+3n $, where we have to find the rank of 
$\binom{r}{2} \approx \frac{r^2}{2!}$ matrices and the computational cost is about 
$$\frac{2}{3} \cdot (r-2)^2 \cdot \binom{r/3 + 4}{2} \cdot \frac{r^2}{2} \approx \frac{r^6}{54}.$$

Thus, for a general set $A$ which verifies the conditions of Proposition \ref{proposition:1}, the total cost of the computation is about 
$$ \frac{59 r^6}{270}.$$ 
\end{remark}

\section{Removing assumptions on the Kruskal's ranks}
\label{Removing assumptions on the Kruskal's ranks}
As we showed in the previous section, the reshaped Kruskal's criterion can be rather demanding from a computational point of view. So we would like to use another strategy to verify that a form is identifiable.
 \smallskip
 
 Through this section, fix $n\geq 0$ and take a form $T$ of degree $d=8+2n$ in three variables, with a decomposition 
 $A=\{P_1,\dots,P_r\}\subset\Pj^2$ of length $r=\ell(A)\leq 11+3n$. \emph{We will always assume that $A$ is non-redundant}, 
 a condition that is easy to check: it suffices to prove that $v_d(A)$ is linearly independent and all the coefficients $a_i$'s of the decomposition are non-zero.
 
 We will prove that $T$ is identifiable unless there are infinitely many decompositions of $T$. 

\begin{remark}
\label{shh}

We will make no hypothesis on how the points of the decompositions $ A $ of $ T $ are located in the plane. We can prove that this method is more efficient than reshaped Kruskal criterion, even when $ A $ is contained in some cubic curves. In fact we do not need the computation of $ k_2 $. 
\end{remark}

We start our analysis with the case in which we cannot find $ 5 +n $ points of $ A $ aligned or $ 9+2n $ points of $ A $ in a conic curve. Then, we will analyse separately the cases in which there are $ 5+n $ points of $ A $ aligned and the case in which there are $ 9+2n $ points of $ A $ contained in a conic curve. Finally, we will use an inductive strategy to prove the main result of the section.

\begin{proposition} Fix $ n \in \N $. Take $ A \subset \Pj^2 $ a decomposition of a form $ T $ in 3 variables of degree $ d $. Suppose that $ r=\ell(A) $ is at most $ 11+3n $ and $ A $ is non-redundant. If $A$ does not contain $ 5+n $ points on a line or $ 9+2n $ points in a conic curve, then $ CB(d) $ cannot hold for $ A\cup B $ where $ B $ is another non-redundant decomposition of $ T $ of length $ \ell(B) \leq r $. Thus, $A\cap B\neq \emptyset$.
\label{claim:2}
\end{proposition}

In order to prove Proposition \ref{claim:2} we will use Theorem 1.5.1 of \cite{undici}. We cite the statement of the theorem below.

\begin{theorem}
	Let $T$ be a form of degree $d$ in $ m +1$ variables, with a non-redundant decomposition $A \subset \Pj^{m}$. If $\ell(A)\leq \frac{d+1}{2}$ then $A$ is minimal
	and $T$ is identifiable.
	\label{theorem:1}
\end{theorem}
\begin{proof} See Theorem 1.5.1 of \cite{undici}.

\end{proof}

Now, we have all we need to prove Proposition \ref{claim:2}.

\vspace{0.5 cm}

\emph{Proof of Proposition \ref{claim:2}}: Fix $ n \in \N $ and set $A= \{ P_1,P_2, \dots P_{r} \}$ a non-redundant decomposition of a form $ T \in \Pj(Sym^d(\C^3))$ contained in a cubic curve such that $ r\leq 11+3n $. 

We analyse the following cases:
\begin{itemize}
	\item[1)] $ r < 5+n $.
	\item[2)] $ 5+n \leq r < 9+2n $. 
	\item[3)] $ 9+2n \leq r \leq 11+3n $.
\end{itemize}
\vspace{5mm}
If $ r < 5+n $ the results follows directly from Theorem \ref{theorem:1}. 

\vspace{5mm}

Suppose $ 5+n \leq r < 9+2n $. From the fact that there are not $ 5+n $ points of $ A $ aligned we know that $ A $ itself is not aligned. Assume that $CB(d)$ holds for $Z=A\cup B$. Thus, by Theorem \ref{GKRext} we have
$$\sum_{j=0}^{1} Dh_Z(j) \leq \sum_{j=2n+8}^{2n+9}Dh_Z(j).$$ 
So, $\sum_{j=0}^{1}Dh_Z(j)\geq \sum_{j=0}^{1}Dh_A(j)=3$ (Lemma \ref{lemma:mix}) then from Proposition \ref{nonincr} we have that $Dh_Z(2n+8)$ 
is at least equal to $2$ otherwise $Dh_A(2n+8)+Dh_A(2n+9)$ would be less or equal to 2. 
From the fact that $A$ is not aligned we know that $Dh_Z(1) \geq Dh_A(1)=2$.
We get:
\begin{multline*} 18+4n > \ell(Z) \geq \sum_{j=0}^{2n+9}Dh_Z(j) = \\ = \sum_{j=0}^{1}Dh_Z(j)+\sum_{j=2}^{2n+7} Dh_Z(j)+\sum_{j=n+8}^{2n+9}Dh_Z(j) \geq
\\ \geq 2 \sum_{j=0}^{1}Dh_Z(j) + \sum_{j=2}^{2n+7} Dh_Z(j) \geq 6 + \sum_{j=2}^{2n+7} Dh_Z(j). \end{multline*}
Thus $\sum_{j=2}^{2n+7} Dh_Z(j) < 12+4n $. But then, by Proposition \ref{nonincr} and from the fact that $Dh_Z(2)\geq 2$ and $Dh_Z(2n+7)\geq 2$, we have:
$$12+4n > \sum_{j=3}^{2n+7}Dh_Z(j) \geq 2\cdot(2n+6)$$
a contradiction.

The second claim follows by Lemma \ref{CBdis} applied to $A,B$ and $T$.

\vspace{5mm}

Suppose now $ 9+2n \leq r \leq 11+3n $. From the fact that $ A $ does not contain $ 9+2n $ points in a conic curve, we know that also $ A $ cannot be contained in any conic curves. 
Thus, we know from the definition of the Hilbert function and Lemma \ref{lemma:mix} that $ Dh_Z(2) \geq Dh_A(2) = 3$. 

Moreover, we know that $ Dh_A(4+n)<3 $. In fact, if $ Dh_A(4+n) \geq 3 $ then we have that:$$ \ell(A) = \sum_{j=0}^{\infty}Dh_A(j) \geq \sum_{j=0}^{4+n}Dh_A(j) \geq 3(n+4)>\ell(A),$$ a contradiction.

If $ Dh_A(4+n)=0 $ then $CB(d)$ cannot hold for $ Z $. 
In fact, suppose by contradiction that $CB(d)$ holds for $ Z $.
Then, by Theorem \ref{GKRext} we have: 
$$ \ell(A) = \sum_{j=0}^{3+n} Dh_A(j) \leq \sum_{j=0}^{3+n} Dh_Z(j) \leq \sum_{j=6+2n}^{9+2n} Dh_Z(j) .$$
From $ \ell(A) \geq \ell(B) $ we have $ 2\ell(A) \geq \ell(Z) $. Thus, we can find the following inequality:
\begin{equation}
2\ell(A) \leq \sum_{j=0}^{3+n} Dh_Z(j) + \sum_{j=6+2n}^{9+2n} Dh_Z(j) \leq \ell(Z).
\label{3}
\end{equation} 

So, we have $ 2 \ell(A)=\ell(Z) $ and as a consequence we have that: $$\sum_{j=0}^{3+n} Dh_Z(j) + \sum_{j=6+2n}^{9+2n} Dh_Z(j) = \ell(Z).$$
We can notice that $ Dh_Z(4+n) $ has to be equal to $ 0 $ and this is a contradiction. In fact, from Proposition \ref{nonincr} we have that if $ Dh_Z(4+n)=0 $ then $ Dh_Z(j)=0 $ for all $ j \geq 4+n $.

We can conclude that $ CB(d) $ does not hold for $ Z $.
As before, the second claim follows by Lemma \ref{CBdis} applied to $A,B$ and $T$.

Suppose now $ Dh_A(4+n) = 1 $. As before, from Proposition \ref{nonincr} we know that also $ Dh_A(3+n) \neq 0 $.

If $ Dh_A(3+n)=Dh_A(4+n)=1 $ from Theorem \ref{BGM} we have that $ 5+n$ points of $ A $ are aligned and so, $ A $ does not satisfies the hypothesis of the proposition.

If $ Dh_A(3+n)=2 $ and $ Dh_A(4+n)=1 $ we have two possibilities: by Proposition \ref{nonincr} $ Dh_A(5+n) $ can be equal either to $ 1 $ or $ 0 $ . 

If $Dh_A(5+n)=1$ , from Theorem \ref{BGM} we have that $ 6+n $ points of $ A $ are contained in a line and so, $ A $ does not satisfies the hypothesis of the proposition. 

If $Dh_A(5+n)=0$ then $CB(d)$ cannot hold for $ Z $. 

In fact, suppose by contradiction that $ CB(d) $ holds for $ Z $.
From the fact that $ A $ is contained in no conic curve, from $ Dh_A(4+n)=1 $ and Proposition \ref{nonincr} we have:\begin{multline*}
	\ell(A)=\sum_{j=0}^{4+n} Dh_A(j) = Dh_A(0)+Dh_A(1)+ Dh_A(2)+\sum_{j=3}^{4+n} Dh_A(j) \geq \\ \geq 1+2+3+2+n=8+n.
\end{multline*} 
Thus by Theorem \ref{GKRext} we have that $$ \ell(A)= \sum_{j=0}^{4+n}Dh_A(j) \leq \sum_{j=0}^{4+n}Dh_Z(j) \leq \sum_{j=5+n}^{9+2n} Dh_Z(j) \leq \ell(Z) .$$

 So, we have that $\sum_{j=5+n}^{9+2n} Dh_Z(j)\geq 8+n$ and as a consequence we have that $ Dh_Z(5+n) >1 $ otherwise we would have by Proposition \ref{nonincr} that $\sum_{j=5+n}^{9+2n}Dh_Z(j)=5+n $ .
 
 Moreover, from the fact that $ \ell(A)\geq \ell(B) $ we have that $ \ell(Z)\leq 2 \ell(A) $ and as a consequence:
 $$ 2\ell(A)\leq \sum_{j=0}^{4+n}Dh_Z(j) +\sum_{j=5+n}^{9+2n}Dh_Z(j)= \sum_{j=0}^{9+2n}Dh_Z(j) \leq \ell(Z).$$
 In particular we have $ 2 \ell(A)=\ell(Z) $ and $ Dh_A(4+n)=Dh_Z(4+n)=1 $. This is a contradiction. In fact, from Proposition \ref{nonincr} we cannot have $ Dh_Z(4+n)=1 $ and $ Dh_Z(5+n)>1 $. So $CB(d)$ cannot hold for $ Z $.

As before, the second claim follows by Lemma \ref{CBdis} applied to $A,B$ and $T$.

If $ Dh_A(4+n)=2 $ we can have either $ Dh_A(3+n)=2 $ or $ Dh_A(3+n)=3 $.
If $ Dh_A(3+n)=2 $ then from Theorem \ref{BGM} we have that $ 9+2n $ points of $ A $ are contained in a conic curve, so the hypothesis of the proposition are not satisfied by $ A $. 

If $ Dh_A(3+n)=3 $ then $CB(d) $ cannot hold for $ Z $.
In fact, suppose by contradiction that $CB(d)$ hold for $ Z $.
As before, from the fact that $ Dh_A(4+n)=2 $, from $ Dh_A(3+n)=3 $ and Proposition \ref{nonincr} we have that $ \ell(A)\geq 11+3n $ so, by hypothesis, we have that $ \ell(A)= 11+3n $.
Moreover, we know from Proposition \ref{nonincr} and Lemma \ref{lemma:mix} that $ Dh_Z(3) \geq Dh_A(3) \geq 3 $. 

Thus, by Theorem \ref{GKRext} we have:
\begin{equation}9 \leq \sum_{j=0}^{3}Dh_Z(j) \leq \sum_{j=2n+6}^{d+1}Dh_Z(j).\label{proof:4} \end{equation}
We know that $Dh_Z(3)\geq Dh_A(3)=3$ and since {\small $\sum_{j=0}^{3}Dh_Z(j)\geq \sum_{j=0}^{3}Dh_A(j)=9$ }
we get $ \sum_{j=2n+6}^{d+1} Dh_Z \geq 9$.
Furthermore, from the fact that $Dh_Z$ is not increasing we have $Dh_Z(2n+6)\geq 3$ (otherwise we would have $\sum_{j=2n+6}^{d+1} Dh_Z \leq 8$, by Proposition \ref{nonincr}).

Using the inequality (\ref{proof:4}), we have
\begin{multline*} 6n+22\geq \ell(Z)\geq \sum_{j=0}^{2n+9}Dh_Z(j) = \\ \sum_{j=0}^{3}Dh_Z(j)+ \sum_{j=4}^{2n+5}Dh_Z(j)+\sum_{j=2n+6}^{2n+9}Dh_Z(j)
\\ \geq 9 + \sum_{j=4}^{2n+5}Dh_Z(j) + 9,\end{multline*}
thus $\sum_{j=4}^{2n+5}Dh_Z(j) \leq 6n +4$. Then, by Proposition \ref{nonincr}, since $Dh_Z(3)\geq 3$ and $Dh_Z(2n+6)\geq 3$, 
we have $\sum_{j=4}^{2n+5}Dh_Z(j) \geq 3(2n+2)$ for some $i$. Thus:
$$ 6n + 4 \geq \sum_{j=4}^{2n+5}Dh_Z(j) \geq 6n+6,$$
a contradiction.

As before, the second claim follows by Lemma \ref{CBdis} applied to $A,B$ and $T$.

\hfil\qed

\medskip
Next, we analyse the behavior of a form $ T $ for which Proposition \ref{claim:2} does not hold. In particular, given a decomposition $ A= \{ P_1,P_2, \dots,P_r\} $ of $ T $, we analyse the following cases:
\begin{itemize}
	\item[1)] There is a subset $ A' $ of $ A $ such that $ A' $ is aligned and $ \ell(A')\geq 5+n $.
	\item[2)] There is a subset $ A' $ of $ A $ such that $ \ell(A')\geq 9+2n $ and $ A' $ is contained in a conic curve.
	
\end{itemize}

In case 1 we are going to prove that $T$ has always an infinite family of decompositions. To do that, we recall the following result.

\begin{proposition}
	\label{proposition:qudis}
	Assume that a decomposition $A \subset \Pj^m$ (not necessarily minimal) of length $\ell(A)=r$ of a form $T$ in $ m+1 $ variables and degree $ d $ is contained in a projective curve $C \in \Pj^m $ 
	which is mapped by $\nu_d$ to a space $\Pj^k$, with $k < 2r-1$. Then there exists positive dimensional family 
	of different decompositions $ \{ A_t \}$ of $T$, such that $A_0 = A$.
\end{proposition}
\begin{proof} A proof of this Lemma can be found in Lemma 34 of \cite{bergimid}. This result is true also in higher dimensional projective spaces but we will use it only for plane curves.
	
\end{proof}

As a consequence, we can prove the following lemma.

\begin{lemma}
	\label{lemma:2}
	Fix $ n \in \N $. Given form $T$ in 3 variables of rank $ r \leq 11+3n $ and degree $ d=8+2n$, and given a decomposition $A= \{ P_1,P_2, \dots , P_{r} \}$ of $T$ , such that there exists a subset $ A' $ of $ A $ with $ \ell(A')=r'\geq 5+n $ and $A'$ is aligned, then there exists a positive dimensional family $ \{ A_t \}$ of decompositions of $T$ such that $A_0=A$.
\end{lemma}
\begin{proof} We may assume that $ A' = \{ P_1, P_2, \dots ,P_{5+n}\} \subset L$ such that $ L $ is a line.
	If $ T= a_1 v_d(P_1)+a_2v_d(P_2)+\dots +a_r v_d(P_r) $ we define $ T_0 $ as follows: $$ T_0=a_1 v_d(P_1)+a_2v_d(P_2)+\dots +a_{5+n} v_d(P_{5+n}).$$
The image of $L$ through $v_d$ is the composition of $v_1$ and $v_d$ applied to $\Pj^1$. Thus, $v_d(C)$ is embedded in a $\Pj^ {d}$. 
	
	Moreover, we have that the inequality
	$$d=8+2n \leq 2r'-1$$ holds for all $r'$ such that $5+n \leq r' \leq 11+3n$. In fact:
	$$2r'-1 \geq 2(5+n)-1 > 8+2n.$$
	Thus, from Proposition \ref{proposition:qudis} $ T_0 $ has an infinite family of decompositions. If we add $ a_{6+n}v_d(P_{6+n})+\dots +a_{r} v_d(P_r) $ to all the decompositions $A'_t$ of $ T_0 $ we find an infinite family of decompositions $A_t$ for $T$.

\end{proof}

Case 2 is similar to case 1. When there are at least $ 9+2n $ points of $A$ contained in a conic we can find the existence of an infinite family of decompositions for $T$. 

\begin{lemma}
	\label{lemma:3}
	Fix $ n \in \N $. Given a form $T$ in 3 variables of rank $ r \leq 11+3n $, and degree $d = 8+2n $ and given a decomposition $A= \{ P_1,P_2, \dots , P_{r} \}$ of $T$, such that there exists a subset $ A' $ of $ A $ with $ \ell(A')=r'\geq 9+2n $ and $A'$ is contained in a conic curve, then there exists a positive dimensional family $ \{ A_t \}$ of decompositions of $T$ such that $A_0=A$.
\end{lemma}
\begin{proof} As before, we may assume that $ A' = \{ P_1, P_2, \dots, P_{9+2n}\} \subset C$ such that $ C $ is a conic curve.

If $ T= a_1 v_d(P_1)+a_2v_d(P_2)+\dots +a_r v_d(P_r) $ we define $ T_0 $ as follows: $$ T_0=a_1 v_d(P_1)+a_2v_d(P_2)+\dots +a_{9+2n} v_d(P_{9+2n}).$$

As before, a conic curve $C$ is a rational normal curve, so it is an image through $v_2$ of $\Pj^1$. Moreover, the image of $C$ through $v_d$ is the composition of $v_2$ and $v_d$ applied to $\Pj^1$. So $v_d(C)$ is embedded in a $\Pj^ {2d}$ and, as before, the inequality $2d=16+4n \leq 2r'-1$ holds for all $r'$ such that $9+2n \leq r' \leq 10+3n$.
In fact:
$$2r'-1 \geq 2(9+2n)-1 > 16+4n. $$
Thus, from Proposition \ref{proposition:qudis}, $ T_0 $ has an infinite family of decompositions. If we add $ a_{10+2n} v_d(P_{10+2n})+\dots +a_r v_d(P_r) $ to all the decompositions $A'_t$ of $ T_0 $ we find an infinite family of decompositions $A_t$ for $T$.

\end{proof}

Now, we are able to describe the behaviour of all forms of degree $r \leq 11+3n$ with a decomposition contained in at least one cubic curve.

\begin{theorem}
\label{proposition:TM}
Fix $ n \in \N $. Given a form $T$ in 3 variables of rank $r= \ell(A) \leq 11+3n$, $ n \in \N $ and degree $ d=8+2n $ such that $A= \{P_1,P_2, \dots ,P_r \}$ is a non-redundant decomposition of $T$, then either $A$ is unique or there is an infinite family $A_t$ of decompositions of length $11+3n$, such that $A_0=A$.
\end{theorem}
\begin{proof} If $r \leq 11+ 3n$, we know from Proposition \ref{claim:2}, Lemma \ref{lemma:2} and Lemma \ref{lemma:3} that the existence of two \emph{disjoint} decompositions implies the existence of an infinite family of decompositions for $ T $.

Suppose now that $r \leq 11+ 3n$, $T$ has not infinitely many decompositions, so $ A $ does not contain $5+n$ points aligned or $9+2n$ points in a conic curve, and suppose that $B$ is another decomposition for $T$ of cardinality $\ell(B)=k \leq r$. 
Of course we may assume that $B$ is non-redundant.
If $A \cap B= \emptyset$ we have a contradiction from Proposition \ref{claim:2}. So, assume $A\cap B \neq \emptyset$. 

Thus we can write, without loss of generality, $B=\{ P_1,\dots P_i,P'_{i+1},\dots,P'_k\}$ i.e.
we may assume that $A\cap B=\{ P_1,\dots, P_i\}$, $i>0$.
Then there are coefficients $a_1,\dots,a_{r},b_1,\dots,b_k$ such that:
\begin{multline*}T=a_1v_d( P_1)+\dots+a_iv_d(P_i)+a_{i+1}v_d(P_{i+1})\dots+a_{r} v_d(P_{r}) = \\
=b_1 v_d(P_1)+\dots+b_i v_d(P_i)+ b_{i+1}v_d(P'_{i+1})+\dots+b_k v_d(P'_k).\end{multline*}
Consider the form 
$$T_0 = (a_1-b_1)v_d( P_1)+\dots+(a_i-b_i)v_d( P_i)+a_{i+1}v_d(P_{i+1})+\dots+a_r v_d(P_{r}),$$
which is also equal to $ b_{i+1}v_d(P'_{i+1})+\dots+b_k v_d(P'_k)$. Thus $T_0$ has two decompositions $A$ and
$B'=\{P'_{i+1},\dots,P'_k\}$, which are disjoint. Thus, if $A$ and $B'$ are both non-redundant, then by Lemma \ref{CBdis} 
applied to $A,B'$ and $T_0$, we get that $A\cup B'$ satisfies $CB(d)$.
Since $A\cup B'=A\cup B=Z$, and we know by Proposition \ref{claim:2} that $Z$ does not satisfies $CB(d)$,
we find that either $A$ or $B'$ are redundant. 

Assume that $B'$ is not non-redundant. Then we can find a point of $B'$, say $P'_{k-1}$, such that
$T_0$ belongs to the span of $v_d(B'\setminus\{P'_k\})$. Since $T=T_0+b_1v_d(P_1)+\dots+b_iv_d(P_i)$,
this would mean that $T$ belongs to the span of $v_d(B'\setminus\{P'_k\})$, which contradicts the fact that $B$ is non-redundant.

Assume that $A$ is not non-redundant, and $T_0$ belongs to the span of $v_d(A\setminus\{P_j\})$,
for some $j>i$. As above, since $T=T_0+b_1v_d(P_1)+\dots+b_r v_d(P_r)$,
this would mean that $T$ belongs to the span of $v_d(A\setminus\{P_j\})$, which contradicts the fact that $A$ is non-redundant.

Assume that $A$ is not non-redundant, and $T_0$ belongs to the span of $v_d(A\setminus\{P_j\})$,
for some $j\leq i$, say $j=1$. Then $T_0=\gamma_2 v_d(P_2)+\dots+\gamma_{r}v_d(P_{r})$,
for some choice of the coefficients $\gamma_j$. Since $v_d(A)$ is linearly independent, because $A$ is non-redundant,
this is only possible if $a_1-b_1=0$. So there exists a proper subset $A'\subset A$ which provides a non-redundant decomposition of $T_0$, together with $B'$. Moreover $\ell(A')\leq 10+3n$ and $ A' \cap B' =\emptyset$.

From the fact that $ A' \subset A $, we also know that $ A' $ does not contain $ 5+n $ points aligned or $ 9+2n $ points in a conic curve.

If there are not $ n+5 $ points of $ A' $ aligned or $ 2n+9 $ points of $ A' $ contained in a conic, all the hypothesis of Proposition \ref{claim:2} are satisfied by $ T_0 $, so we have that $T_0$ cannot have two disjoint decompositions. So the existence of $ B' $ yields a contradiction.

 We conclude that $ T $ is identifiable.

\end{proof}

This results cannot be extended to higher value of $ r $. In fact, we can find an example of a form $ T $ of rank $ 12 $ in degree $ 8 $ that is not identifiable and for which there are exactly two decompositions. This follows from a well known results proved by Ciliberto and Chiantini in \cite{CCi06}.

\begin{example}
Take $ T $ a form in 3 variables and $ A=\{P_1,P_2,\dots ,P_{12} \} $ a decomposition of $ T $ in degree $ 8 $ contained in an unique irreducible, smooth plane cubic curve $ C $. This case is outside our numerical bound for the length of the decomposition. We claim that $T$ has two different decompositions (so that our range is sharp). The proof is the same of Theorem 5.1 of \cite{CCi06} and it is a direct consequence of Theorem 2.4, Theorem 2.10 and Proposition 5.2 of \cite{CCi06}.

From our point of view, we can prove the claim as follows.
From the fact that $ A $ is contained in a unique irreducible, smooth cubic curve $ C $, we know from Theorem \ref{GKRext} and from Theorem \ref{BGM} that all the other decompositions $B$ of $T$ lie in $C$ and, moreover, the function $Dh_Z$ of $Z=A \cup B$ is symmetric around degree $5$, i.e. it is:
\smallskip
\begin{center}\begin{tabular}{c|ccccccccccc}
		$i$ & 0 & 1 & 2 & 3 & 4 & 5 & 6 & 7 & 8 & 9 & 10 \dots \\ \hline
		$Dh_Z(i)$ & 1 & 2 & 3 & 3 & 3 & 3 & 3 & 3 & 2 & 1 & 0 \dots
	\end{tabular} 
\end{center}
This implies that $ Z $ is a complete intersection of the cubic $ C $ with a curve $ C' $ of degree $ 8 $ (see the main theorem of \cite{Dav84}). Moreover, we obtain that $\ell(B)$ is $12$. Furthermore, the intersection of the span of $v_8(A)$ with the span of $v_8(B)$ is only $T$, because $\sum_{i=9}^\infty (Dh_Z(i))=1$ (see section 6 of \cite{otto}). One computes that the sets of $12$ points $B$ in the plane which, together with $A$, are a complete intersection of type $3,8$, are parametrized by a projective space of dimension $11$. This space maps birationally to the span of $P_1,P_2, \dots,P_{12}$ (this can be obtained by a direct computation on one specific point $T$, see e.g. \cite{AC19} Claim 4.4). Thus a general $T$ in the span of $P_1,\dots ,P_{12}$ has two decompositions. 
\end{example}
Notice that, numerically, if we know that $A \cap B = \emptyset$ and $A$ is not contained in a cubic curve, then we can conclude the identifiability. Thus, In order to repeat the proof of Theorem \ref{proposition:TM} for $\ell(A)=3n+12$ we must have that no subset of 10 points of $A$ sits in a cubic curve. To control that, we need to compute $k_3$, which we would like to avoid, to maintain a cheap computational cost.

\section{Excluding the existence of a family of decompositions}
\label{Excluding the existence of a family of decompositions}
The main difference between Theorem \ref{proposition:TM} and Proposition \ref{proposition:1} is that in order to check the hypothesis of 
Theorem \ref{proposition:TM} we do not need to compute the Kruskal's ranks $k_2$ and $k_{n+3}$, but only to determine the non existence 
of a family of decompositions ${A_t}$ for $T$.
This can be done by means of the \emph{Terracini's test} on $A$.

We will need a series of definitions related to the secant varieties of a Veronese embedding.
We refer to section 5 of \cite{dodici} for the proofs of the claims below.\\
Denote with $\Sigma_r$ the closure of the subset of forms of rank $r$ in \begin{small}
$\Pj(Sym^d(\C^{m+1}))$
\end{small} 
and denote with $(\Pj^m)^{(r)}$ the symmetric product.
 
We define the \emph{abstract secant variety}, and the \emph{$r-$secant map}, as follows.

\begin{definition} We define the abstract secant variety $A\Sigma_r$ as the subvariety of $\Pj(Sym^d(\C^{m+1})) \times (\Pj^m)^{(r)}$ which is the Zariski closure of the set of pairs \begin{small}$(T,[P_1, \dots, P_r ])$ \end{small} such that the set $\{v_d(P_1), \dots , v_d(P_r) \}$ spans a subspace of dimension $r - 1$ 
in $\Pj(Sym^d(\C^{m+1}))$ and $T$ belongs to the span of $v_d(A)$.

We define the $r-$th secant map $s_r$ as the first projection
$$ s_r : A \Sigma_r \rightarrow \Pj(Sym^d(\C^{m+1})).$$
\end{definition}

We note that the image of the secant map is $\Sigma_r$ and that the inverse image of a form $T$ in the secant map is the set of decompositions 
of cardinality $r$ of $T$.
Furthermore, since $\Pj^m$ is a smooth variety, then $(\Pj^m)^{(r)}$ is smooth outside the diagonals.
Thus, if $U$ is the open set of $(\Pj^m)^{(r)}$ of sets $[P_1,\dots, P_r]$ such that $\{v_d(P_1), \dots , v_d(P_r) \}$ is linearly independent,
then $A\Sigma_r$ is a $\Pj^{r-1}$ bundle over $U$, thus $s_r^{-1}(U)$ is smooth, of dimension $(r-1+rm)$ ($=\dim A\Sigma_r$).

We can now define the Terracini's space $\tau$.

\begin{definition} We call the Terracini's space of a decomposition \begin{small} $A = \{ P_1,\dots ,P_r \}$ \end{small}of $T$ the image of the tangent space to $A\Sigma_r$ 
at the point $(T, [\{ P_1,\dots ,P_r])$ in the differential of $s_r$. Thus $\tau$ is a linear subspace of $\Pj^N = \Pj(Sym^d(\C^{m+1}))$.
 It is the linear space spanned by the tangent spaces to $v_d(\Pj^m)$ at the points $v_d(P_1),\dots ,v_d(P_r)$. 
\end{definition} 

The Terracini's Lemma (see \cite{diciannove}) says that for a general choice of $T\in \Sigma_r$ and for $r\leq N$
the Terracini's space is the tangent space to $\Sigma_r$ at $T$.

\begin{remark} The dimension of the Terracini's space $\tau$ is naturally bounded:
$$dim(\tau ) \leq (m + 1)r - 1$$
and the equality means that the tangent spaces to $v_d(\Pj^m)$ at the points $v_d(P_i)$'s are linearly independent.
By \cite{ventidue}, as soon as $d>2$ and $r>5$ we know that for a \emph{general} choice of the set $A$, 
the dimension of the Terracini's space equals the expected dimension.
\end{remark}

The main link with our problem is given by the following observation.

\begin{proposition}\label{nofam} Let $A \subset \Pj^m$ be a non-redundant decomposition of a form $T$ in $ m+1 $ variables such that $ \ell(A)=r$. Assume also that there exists a non trivial family $A_t$ of decompositions of $T$, such that $A_0$ = $A$. 
Then the Terracini's space $\tau$ of $A$ has dimension strictly smaller than $(m + 1)r - 1$.
\end{proposition}
\begin{proof} $A_t$ determines a positive dimensional subvariety $W$ in the fiber of $s_r$ over $T$ (see Proposition 2.3 of \cite{sette}). 
Thus, there exists a tangent vector to $A\Sigma_r$ at $(T, [A])$, where $[A]$ is the point of the symmetric 
product corresponding to $A$, which is killed by the differential of $s_r$ at $(T, [A])$. 
So $\tau$ can not have his maximal dimension and this conclude the proof.

\end{proof}

Unfortunately, the converse of the previous proposition does not hold in general.
Yet, the proposition implies that in order to exclude the existence of the family, it is sufficient to control that the dimension of the Terracini's space of $A$
attains the expected value. We can collect our results in the following.

\begin{theorem} Fix $ n \in \N $. Let $A$ be a non-redundant decomposition of a form $T$ in three variables of degree $d=2n+8$ with $ n \in N $. Assume that $\ell(A)\leq 3n+11$. If the dimension of the Terracini's space of $A$ equals the expected dimension $3r-1$,
then $A$ is minimal. Thus $T$ has rank $r=\ell(A)$ and it is identifiable.
\end{theorem}

In the following remark, we explain how the dimension of the Terracini's space can be computed, in practice. 
The claims below on the structure of tangent spaces to Veronese embedding
are standard, and can be found e.g. in \cite{ventitre} (see also \cite{sette}).

\begin{remark}
As we noticed in the introduction, a decomposition \begin{small} $A= \{ P_1, \dots ,P_r \}$ \end{small} of $T \in \Pj(Sym^d(\C^{m+1}))$ 
corresponds to the datum of $r$ linear forms $L_1,L_2, \dots, L_r$ in three variables. 

The tangent spaces to $v_d(\Pj^m)$ at $v_d(P_i)$ can be identified with the degree $d$ homogeneous piece of the ideal spanned by 
$ L^{d-1}_i m$, where $m$ is the ideal generated by the variables.

It follows that the Terracini's space can be identified with the degree $d$ homogeneous piece of the ideal spanned by 
$$ L^{d-1}_1 m, \dots ,L^{d-1}_1 m.$$ 

 Thus, in our case, we have that the Terracini's space is the ideal spanned by 
$L_1^{7+2n}m, \dots L_r^{7+2n}m$ with $ n \in \N $.

The computation of the dimension of the $d$-th piece of this ideal corresponds to the computation of the rank of the matrix of 
coefficients of the forms $L_i^{7+2n}x_j$, where the $x_j$'s are the coordinates of $\Pj^2$.
\end{remark}

We are now able to write an algorithm for detecting the identifiability of $T$.

\subsection{The algorithm}
\label{The algorithm}
Let $T$ be a form in three variable of degree $d=8+2n$ with $ n \in \N $.

Assume that we are given a decomposition $A$ of $T$, $A= \{P_1, P_2, \dots , P_r \} \subset \Pj^2.$ Assume that $\ell(A)\leq 3n+11$ and
that, if $T= \{ a_1v_d{P_1}+a_2v_d{P_2}+\dots +a_rv_d{P_r} \}$ none of the $ a_i $ are equal to zero.
Then, in order to prove that $A$ is minimal and so $T$ is identifiable, the following steps could be taken.
\begin{itemize}
\item[S0.] Compute the rank of the matrix $M_d$ of coordinates of the points $v_d(P_j)$'s.
\begin{itemize}
\item[S0.1] If the rank of $M_d$ is smaller than $r$, then $A$ is redundant, and the algorithm 
terminates and states that $T$ has rank $<r$. 
\item[S0.2] If the rank of $M_d$ is $r$, then $A$ is non-redundant, and the algorithm continues.
\end{itemize}
\item[S1.] If $r \leq 4 +n$, the algorithm terminates and states that $T$ is identifiable.
\item[S2.] Perform the Terracini's test as follows:
\begin{itemize}
\item[S2.1] Compute the linear forms $L_1, \dots ,L_r$ associated with $P_1,\dots , P_r$.
\item[S2.2] For $i=1,2, \dots r$ and $j=0,1,2$ compute the rank of the matrix of coefficients of the forms $x_j L_i^{7+2n} $, and call it $q$
\item[S2.3] If $q < 3r$ then the algorithm terminates claiming that it cannot prove the identifiability of $T$. 
\item[S2.4] If $q = 3r$ the algorithm terminates and states that $T$ is identifiable. 
\end{itemize}

\end{itemize}

Now we can show that using the method exposed below we can reduce considerably the computational cost.

\begin{remark} \label{compucost}
In order to find the dimension of the Terracini's space, the crucial step is
 to compute the rank of the matrix made by $[x_i (L_j)^{7+2n}]$ with $i=1,2,3$ and $j=1, \dots, r$. 
 So, we have to compute the rank of a $ \binom{9+2n}{2} \times 3r$ matrix. 
 Using the Gauss elimination method, we have that the computational costs of this process is in the order of 
 $$\frac{2}{3} \cdot \frac{(2r/3+27/22)^2}{2} \cdot 9r^2 \approx \frac{4}{3}r^4.$$

Notice indeed that to verify that $A$ is contained in a cubic curve, so to compute $h_A(3)$, we have to find the rank 
of the matrix $[v_3(P_1), \dots , v_3(P_r)]$. With the Gauss elimination method,
we have a computational cost in the order of $$ \frac{2}{3} 10^2 \cdot r. $$

So, the total computational cost is in the order of $$\frac{4}{3}r^4.$$

Then, comparing the two method, we have that Terracini's test can be much quicker then computing the Kruskal's ranks for high value of $r$.
\end{remark}

\section{Acknowledgments}
The author would like to thank Luca Chiantini for several fruitful discussions on the topics of the present research and his precious advices.

The author would like to thank also the anonymous referee for the useful observations about this work.

The author is a Ph.D. student at Universit\`a degli Studi di Siena, Dipartimento di Ingegneria dell'Informazione e Scienze Matematiche and he is supported by a Ph.D. grant for the Ph.D. program Information Engineering and Science.

\nocite{*} 
\bibliographystyle{ieeetr}
\bibliography{prova}

\end{document}